\documentclass[11pt]{amsart}
\usepackage{amsmath,amssymb,amsthm}
\usepackage[alphabetic]{amsrefs}
\usepackage{thmtools,thm-restate}
\usepackage{nameref,hyperref,cleveref}
\usepackage{tikz-cd}
\usepackage{comment}
\newcommand{\eps}{\varepsilon}
\newcommand{\C}{{\mathbb C}}

\newcommand{\D}{{\mathbb D}}

\newcommand{\R}{{\mathbb R}}
\newcommand{\Q}{{\mathbb Q}}
\newcommand{\Z}{{\mathbb Z}}
\newcommand{\cB}{{\mathcal B}}

\newcommand{\cF}{{\mathcal F}}
\newcommand{\cJ}{{\mathcal J}}
\newcommand{\cK}{{\mathcal K}}

\newcommand{\cU}{{\mathcal U}}
\newcommand{\cV}{{\mathcal V}}

\DeclareMathOperator{\dist}{dist}

\theoremstyle{plain}
\declaretheorem[numberwithin=section]{theorem}
\declaretheorem[sibling=theorem]{proposition}

\declaretheorem[sibling=theorem]{corollary}
\declaretheorem[sibling=theorem]{conjecture}

\declaretheorem[style=definition,sibling=theorem]{definition}
\declaretheorem[numbered=no,style=definition]{remark}
\declaretheorem[numbered=no,style=definition]{example}

\numberwithin{equation}{section}


\begin{document}
\title{Linearizability of Saturated Polynomials}
\author{Lukas Geyer}
\address{Montana State University \\ Department of Mathematical
  Sciences \\ Bozeman, MT 59717--2400 \\ USA}
\email{geyer@montana.edu}
\maketitle

\begin{abstract}
  Brjuno and R\"ussmann proved that every irrationally indifferent
  fixed point of an analytic function with a Brjuno rotation number is
  linearizable, and Yoccoz proved that this is sharp for quadratic
  polynomials. Douady conjectured that this is sharp for all rational
  functions of degree at least 2, i.e., that non-M\"obius rational
  functions cannot have periodic Siegel disks with non-Brjuno rotation
  numbers. We prove that Douady's conjecture holds for the class of
  polynomials for which the number of infinite tails of critical
  orbits in the Julia set equals the number of irrationally
  indifferent cycles. As a corollary, Douady's conjecture holds for
  the families of polynomials $P(z) = z^d + c$ and $Q(z) = z +
  cz^{d-1} + z^d$.
\end{abstract}

\tableofcontents

\section{Introduction and Statement of Results}
Let $f(z) = \lambda z + O(z^2)$ be an analytic function defined in a
neighborhood of zero in the complex plane. We say that $f$ is
\emph{linearizable} near the fixed point $0$ if it is locally
analytically conjugate to its linear part, i.e., if there exists an
analytic function $h(z) = z + O(z^2)$ and $\eps > 0$ such that
$f(h(z)) = h(\lambda z)$ for $|z|<\eps$.

In 1884, Koenigs proved that in the case $|\lambda| \notin \{ 0, 1\}$
of attracting or repelling fixed points every analytic function is
linearizable \cite{koenigs1884}. If $\lambda = 0$, linearizability
obviously implies that $f$ is constant, and in the non-constant
super-attracting case $f(z) = a_m z^m + O(z^{m+1})$, with $m \ge 2$,
$a_m \ne 0$, the function is always locally conjugate to the power map
$z \mapsto z^m$. (This theorem is usually attributed to Boettcher who
first stated it with a sketch of a proof in 1904. The first complete
proof was given by Ritt in \cite{ritt1920}.) In the rationally
indifferent (or parabolic) case where $\lambda = e^{2\pi i p/q}$ is a
root of unity, it is easy to see that the function is linearizable iff
the $q$-th iterate $f^q$ is the identity in a neighborhood of $0$.
For globally defined
functions such as polynomials, rational functions, entire or
meromorphic functions, this never happens unless the function is a
fractional linear transformation. The general question of classifying
local normal forms in the rationally indifferent case turns out to be
complicated, and it has been completed by \'Ecalle \cite{ecalle1981}
and Voronin \cite{voronin1981}.

In this paper we are going to be concerned with the irrationally
indifferent case where $\lambda = e^{2\pi i \alpha}$ with $\alpha$
irrational. In this situation, there are quite a few results, but
several open questions remain. It turns out that the question of
linearizability is closely tied to number-theoretic properties of
$\alpha$. Denote by $p_n/q_n$ the \emph{convergents} of $\alpha$,
i.e. the best rational approximations, obtained by truncating the
continued fraction expansion of $\alpha$. The set of \emph{Brjuno
  numbers} is defined as $\cB = \{ \alpha \in \R \setminus \Q : \sum
q_n^{-1}\log q_{n+1} < \infty \}$.  After earlier results about
linearizability by Cremer and Siegel, the following theorem combines
the results from \cite{Russmann1967}, \cite{Brjuno1971}, and
\cite{Yoccoz1995}.
\begin{theorem}[R\"ussmann, Brjuno, Yoccoz]
  \label{thm:RBY}
  If $\alpha \in \cB$, then every germ $f(z) = e^{2\pi i \alpha} z +
  O(z^2)$ is linearizable. If $\alpha \in \R \setminus \cB$, then the
  quadratic polynomial $P(z) = e^{2\pi i \alpha} z + z^2$ is not
  linearizable.
\end{theorem}
By passing to the appropriate iterate and conjugating, periodic points
of analytic maps in one complex dimension are handled and classified
similarly. For polynomials and rational functions of degree at least
2, irrationally indifferent periodic points are contained in the Fatou
set iff they are linearizable. The Fatou components containing them
are always simply connected and are called \emph{Siegel disks}. The
first part of the previous theorem shows that irrationally indifferent
periodic points with Brjuno rotation numbers are always centers of
Siegel disks. Let us call a Siegel disk \emph{exotic} if its rotation
number is not Brjuno. Note that exotic Siegel disks of polynomials and
rational functions of degree at least 2 (if they exist) will always
have irrational rotation numbers. With this definition, we can state
the biggest open conjecture concerning linearizability, originally
posed in \cite{Douady1986}.
\begin{conjecture}[Douady]
  Polynomials and rational functions of degree $d\ge 2$ do not have
  exotic Siegel disks.
\end{conjecture}
Yoccoz's result shows that quadratic polynomials do not have fixed
exotic Siegel disks, and a simple renormalization argument shows that
they cannot have periodic exotic Siegel disks either,\footnote{If a
  quadratic polynomial $f$ has a Siegel periodic orbit of period $p$
  with multiplier $\lambda$, then $f^p$ is renormalizable, so it is
  topologically conjugate to a quadratic polynomial with a fixed
  Siegel disk with multiplier $\lambda$.  The renormalization is
  constructed using rational external rays and potentials, for details
  see e.g.\ \cite{Milnor2000}.}so the conjecture is true for quadratic
polynomials.  However, it is still open even for cubic polynomials and
for quadratic rational functions.  There are very strong results about
``generic'' polynomials, see \cite{PerezMarco1993} and
\cite{PerezMarco2001}, as well as several results about families or
maps satisfying certain special conditions, see \cite{Geyer1999},
\cite{Geyer2001}, \cite{Okuyama2001}, \cite{Okuyama2005}, and
\cite{Cheraghi2010}.

The following theorem is the main result of this paper, establishing
Douady's conjecture for a class of polynomials.
\begin{restatable*}{theorem}{JuliaSaturatedTheorem}
  Julia-saturated polynomials do not have exotic Siegel disks.
\end{restatable*}
Here a polynomial is \emph{Julia-saturated} if the number of infinite
tails of critical orbits in the Julia set equals the number of
irrationally indifferent cycles. Here a \emph{critical orbit} of a
polynomial $f$ is the
forward orbit $\{ f^n(c) : n \ge 1 \}$ of a critical point $c$ of $f$,
and a \emph{critical orbit tail} is an equivalence class of critical
orbits where two critical orbits are equivalent if their intersection
is non-empty. It is easy to check that this is indeed an equivalence
relation, and that two equivalent critical orbits are either both
finite (corresponding to periodic or preperiodic critical orbits) or
both infinite. An \emph{infinite critical orbit tail} is an
equivalence class of infinite critical orbits.

As an illustration of this somewhat technical condition, the following
corollary shows Douady's conjecture for two explicit families of
polynomials.
\begin{restatable*}{corollary}{MainCorollary}
  There are no exotic Siegel disks in the families $P_{c,d}(z) = z^d +
  c$ and $Q_{c,d}(z) = z + cz^{d-1} + z^d$ for $d \ge 2$ and $c \in
  \C$.
\end{restatable*}

In \cite{BuffCheritat2011}, Buff and Ch\'eritat gave a different proof
of the non-existence of fixed exotic Siegel disks under the condition
that the number of infinite tails of critical orbits equals the number of
indifferent periodic cycles. While this condition is more restrictive
than being Julia-saturated, and their paper does not address periodic
(non-fixed) exotic Siegel disks, their methods give very sharp bounds
on the sizes of Siegel disks in the case of Brjuno rotation numbers.

This paper is structured as follows. \Cref{sec:background-notation}
reviews some background and notation, in \cref{sec:quadr-pert-line} we
introduce concepts of certain polynomial perturbations and uniform
linearizability, and prove our central technical result,
\autoref{prop:maxprinc}. It says that if a linearizable germ $f(z) =
\lambda z + O(z^2)$ admits an essentially quadratic and uniformly
linearizable perturbation, then the quadratic polynomial $P(z) =
\lambda z + z^2$ is linearizable. Here an essentially quadratic
perturbation is a generalization of perturbations of the form $f_a(z)
= f(z) + az^2$ which have been used by Yoccoz and P\'erez-Marco in
similar contexts before. \Cref{sec:analyt-famil-polyn} contains some
background material and folklore results about polynomial-like maps
and analytic families of them, and \cref{sec:j-stability} reviews the
concept of $J$-stability for polynomial-like maps and has a
straightforward result relating $J$-stability, irrationally
indifferent periodic points and uniform linearizability,
\autoref{cor:JStableUniformlyLin}. \Cref{sec:fatou-shish-ineq}
contains two strong versions of the Fatou-Shishikura inequality,
relating the number of non-repelling cycles to the number of critical
orbit tails. Both of these follow from results by Kiwi
\cite{Kiwi2000}, combined with some standard renormalization
techniques. \Cref{sec:satur-polyn} defines saturated and
Julia-saturated polynomials as those polynomials for which equality
holds in one of the versions of the Fatou-Shishikura inequality. It
then combines the results from all previous sections in order to show
that these classes of polynomials do not have exotic Siegel disks. For
saturated polynomials one can explicitly write down a perturbation
which is essentially quadratic for every irrationally indifferent
periodic point and for which the strong form of the Fatou-Shishikura
inequality establishes $J$-stability. Lastly, Julia-saturated
polynomials can be turned into saturated polynomials with some
standard renormalization techniques and results by McMullen from
\cite{McMullen1988}.

I am indebted to the anonymous referee for pointing out some
inconsistencies in the definition of persistent indifferent periodic
points in the literature. In the appendix we address this and related
technical problems related to $J$-stability, as well as the way we
chose to deal with them in our paper.

{\bf Acknowledgments.} I would like to thank Christian Henriksen for
valuable discussions about an early version of the proof, and Joseph
Manlove for many helpful questions and suggestions about all parts of
this paper. I am also indebted to the referee whose suggestions
greatly improved the paper. This work evolved over a long time period
at the Mittag-Leffler Institute in Stockholm, the University of
Michigan, and Montana State University. I would like to thank all
three institutions for providing a great research environment, and I
am particularly grateful for funding from the Mittag-Leffler Institute
and the Alexander von Humboldt foundation.

\section{Background and Notation}
\label{sec:background-notation}
We assume that the reader is familiar with the basics of complex
dynamics as covered in \cite{carlesongamelin1993} or
\cite{Milnor2006}. In this section will review a few basic facts
and explain our notation.

For a polynomial $P$ of degree $d \ge 2$ we will denote the
\emph{Julia set} by $J(P)$, the \emph{filled-in Julia set} by $K(P)$,
and the \emph{basin of infinity} by $A_\infty(P)$. For a point
$z \in \C$ the \emph{forward orbit} is defined as
$O_P^+(z) = \{ P^n(z): n \ge 1 \}$, and its \emph{$\omega$-limit set}
$\omega_P(z)$ is defined as the set of all limits of sequences of
iterates $P^{n_k}(z)$ for sequences $n_k \to \infty$. A tuple of
distinct points $Z=(z_1, \ldots, z_q)$ is a \emph{periodic cycle of
  (minimal) period $q$} if $P(z_k) = z_{k+1}$ for
$k = 1, \ldots, q-1$, and $P(z_q) = z_1$. Each point in the cycle is a
\emph{periodic point of (minimal) period $q$}.  The \emph{multiplier}
of the cycle is $\lambda(Z) = (P^q)'(z_1) = \prod_{k=1}^q
P'(z_k)$. The cycle $Z$ is \emph{super-attracting} if $\lambda = 0$,
\emph{attracting} if $0<|\lambda|<1$, \emph{indifferent} if
$|\lambda| = 1$, and \emph{repelling} if $|\lambda|>1$. In the
indifferent case we say that $Z$ is \emph{rationally indifferent} (or
\emph{parabolic}) if $\lambda$ is a root of unity, \emph{irrationally
  indifferent} otherwise. Irrationally indifferent cycles are either
contained in the Fatou set, in which case the $q$-th iterate $P^q$ is
locally conjugate to a rotation near every point in the cycle, or they
are contained in the Julia set. In the first case we call $Z$ a
\emph{Siegel cycle}, in the second case a \emph{Cremer cycle}. We will
also use all of these terms for the periodic points $z_1, \ldots, z_q$
in the cycle, as well as the periodic Fatou components associated to
those (except in the repelling and Cremer case which do not have
associated Fatou domains.)  E.g., a periodic Siegel point is a
periodic point in an irrationally indifferent cycle contained in the
Fatou set, and a parabolic Fatou component is a periodic component of
the Fatou set on which the iterates converge locally uniformly to a
parabolic periodic cycle. A parabolic Fatou component associated to a
parabolic cycle $Z$ is also called a \emph{petal} associated to $Z$.

Let $Z$ be a parabolic periodic cycle of period $q \ge 1$ of the
polynomial $P$ with multiplier $\lambda = e^{2\pi i s/t}$ where $s,t$
are relatively prime integers with $t \ge 1$. Then $P^{tq} (z) = z_1 +
a_{m+1} (z-z_1)^{m+1} + O(|z-z_1|^{m+2})$ with $a_{m+1} \ne 0$ for
some $m = t r$, where the positive integer $r = r(Z)$ is the number of
invariant cycles of petals attached to the cycle $Z$.  Lastly, we
define $\tau(Z) = m+1$ as the \emph{tangency index} and $tq$ as the
\emph{tangency period} of $Z$. (These definitions are not standard,
but they prove to be useful later on.)

\section{Quadratic Perturbations and Linearizability}
\label{sec:quadr-pert-line}
The main tool to show non-existence of exotic Siegel disks in this
paper is the existence of certain algebraic perturbation families with
(uniformly) persistent Siegel disks.

In order to motivate the following definition, let us first sketch the
general idea in the simple case of one irrationally indifferent fixed
point at zero, basically following the argument from
\cite{PerezMarco1993}: Assume that $f(z) = \lambda z + O(z^2)$ has a
Siegel disk at $0$, and that we can find a perturbation family $f_a(z)
= f(z) + az^2 g(z)$ with $g(0) = 1$ for which we know that the Siegel
disk persists for $|a|$ sufficiently small. The idea now is to
consider this perturbation for large $|a|$, in particular for $|a| \to
\infty$. In order to obtain a sensible limit, we have to rescale $f_a$
by conjugation and introduce the family $F_b(w) = b^{-1}f_{1/b}(bw) =
b^{-1} f(bw) + w^2 g(bw)$ which extends analytically to $b=0$, with
$F_0(w) = \lambda w + w^2$ the quadratic polynomial. From the
construction we know that the family $F_b$ has a Siegel disk at $0$
for sufficiently large $|b|$, and finally a maximum principle for
linearization then shows that the quadratic polynomial $F_0$ also has
a Siegel disk. By Yoccoz's result (\Cref{thm:RBY}) this implies that
$\lambda = e^{2\pi i \alpha}$ with $\alpha \in \cB$, i.e., that the
Siegel disk for $f$ is not exotic.

The following definition is a slightly more flexible version of this
construction. The main reason to introduce the added flexibility is to
deal with periodic points without having to pass to iterates.
Notice that in all perturbations we consider, the constant
and linear parts remain unchanged.

\begin{definition}
  An \emph{analytic perturbation} of an analytic function $f$ defined
  in a neighborhood of $z_0 \in \C$ is an analytic function of two
  variables
  \[
    f(a,z)=f(z_0) + f'(z_0)(z-z_0) + \sum\limits_{n=2}^\infty
    f_n(a)(z-z_0)^n,
  \]
  defined in some neighborhood of $(0,z_0)$, with $f(0,z) = f(z)$,
  where the coefficients $f_n$ are analytic functions of $a$. It is an
  \emph{admissible perturbation} if the coefficients $f_n$ are
  polynomials of degree $d_n < n$. It is a \emph{quadratic
    perturbation} of $f$ at $z_0$ if $d_2=1$ and $d_n\leq 1$ for
  $n>2$.  It is an \emph{essentially quadratic perturbation} of $f$ if
  $d_2=1$ and $d_n<n-1$ for $n>2$.  It is a \emph{sub-quadratic
    perturbation} of $f$ if $d_n<n-1$ for $n\geq 2$.
\end{definition}

If $f$ and $g$ are two analytic perturbations at $z_0$ and $w_0$,
respectively, with $f(z_0)=w_0$ and $g(w_0)=\zeta_0$, then their
composition $g\circ f$, defined as $(g\circ f)(a,z) = g(a,f(a,z))$ is
again an analytic perturbation, analytic in some polydisk centered at
$(0,z_0)$, possibly smaller than the domain of $f$. As a technical
convention, when we talk about functions being analytic in some not
necessarily open set (e.g., a closed polydisk), we mean that the
function is analytic in a neighborhood of the set.

The following proposition shows that the admissible perturbations
$f(a,z)$ are exactly those for which the (rescaled) perturbation limit
for $a \to \infty$ exists.
\begin{proposition}
  \label{prop:flipped-perturbations}
  Let $f(a,z)$ be an analytic perturbation family, analytic for
  $|a| \le r$, $|z-z_0| \le \eps$, and define
  \begin{equation}
    \label{eq:flipped-perturbation}
    F(b,w) = \frac{f(b^{-1}, z_0 + bw)-f(z_0)}b = f'(z_0)w +
    \sum_{n=2}^\infty {F_n(b) w^n}
  \end{equation}
  for $|b| \ge 1/r$, $|w| \le \eps / |b|$. Then $F$ extends
  analytically to a polydisk $|b| \le 1/r$, $|w| \le \delta$ iff $f$
  is an admissible perturbation. In this case, $F(b,w)$ is an
  admissible perturbation of the analytic function $F(w) = F(0,w) =
  f'(z_0) w + O(w^2)$. In particular, the coefficients $F_n(b)$ are
  polynomials in $b$. Furthermore, $F$ is a quadratic polynomial iff
  $f(a,z)$ is an essentially quadratic family, and it is a linear
  function or a constant iff $f(a,z)$ is a subquadratic family.
\end{proposition}
\begin{remark}
  We call the family $F(b,w)$ the \emph{flipped perturbation family}
  and the function $F(w)$ the \emph{perturbation limit} associated to
  $f(a,z)$. It is obvious that we can recover the original
  perturbation family $f(a,z)$ from the flipped family together with
  the points $z_0$ and $f(z_0)$.  Note that in the case of a fixed
  point $z_0 = f(z_0)$, the function $w \mapsto F(b,w)$ near $0$ is
  locally conjugate to $z \mapsto f(a,z)$ near $z_0$, where the
  perturbation parameters $a$ and $b$ are related by $ab=1$. The
  function $F(w)$ in this case is a rescaled limit of the perturbation
  $f(a,z)$ for $a \to \infty$.
\end{remark}
\begin{proof}
  From the definition of the flipped perturbation we know that $F_n(b)
  = b^{n-1}f_n(b^{-1})$. The permanence principle gives that any
  analytic continuation of $F$ to a polydisk $|b| \le 1/r$, $|w| \le
  \delta$ has to be given by the same series
  \eqref{eq:flipped-perturbation} with $F_n(b) = \frac{1}{n!}
  \frac{\partial^n F}{\partial w^n}(b,0)$, so it can only be possible
  if each $F_n$ extends analytically to $|b| \le 1/r$, and it $F_n(0)
  = \lim\limits_{b\to 0} b^{n-1}f_n(b^{-1})$ exists for every $n \ge
  2$. This implies that $f_n$ is a polynomial of degree $d_n < n$,
  i.e., that $f(a,z)$ is an admissible perturbation. It also implies
  that $F_n$ is a polynomial of degree $D_n<n$, so $F(b,w)$ is an
  admissible perturbation of $F(w) = F(0,w)$.

  On the other hand, if $f(a,z)$ is an admissible perturbation, then
  first of all $f$ is bounded on the closed polydisk $|a| \le r$,
  $|z-z_0| \le \eps$, so by Cauchy estimates there exists a constant
  $C$ with $|f_n(a)| \le C \eps^{-n}$ for all $|a| \le r$. This shows
  that $|F_n(b)| \le r^{-n+1} C \eps^{-n}$ for $|b| = 1/r$, and the
  maximum principle then gives the same inequality for all $|b| \le
  1/r$. From this we get that the series for $F(b,w)$ converges
  uniformly for $|b| \le 1/r$ and $|w| \le r\eps$, and hence defines
  an analytic function in $|b| < 1/r$, $|w| < r\eps$, providing the
  claimed analytic extension. The power series coefficients $P_n =
  \lim\limits_{b\to 0} b^{n-1} f_n(b^{-1})$ for the perturbation limit
  $P(w) = F(0,w) = \sum\limits_{n=1}^\infty P_n w^n$ are given by the
  $(n-1)$-st coefficients of $f_n(a)$. This shows that $P$ is a
  quadratic polynomial iff $d_2 = 1$ and $d_n < n-1$ for $n>2$, which
  is exactly the definition of an essentially quadratic perturbation
  family. It also shows that $P$ is linear or constant iff $d_n < n-1$
  for all $n \ge 2$, i.e., if $f(a,z)$ is a sub-quadratic perturbation
  family.
\end{proof}
\begin{proposition}
  \label{prop:composition} 
  Let $f$ and $g$ be admissible perturbations at $z_0$ and
  $w_0=f(z_0)$. Then the composition $h = f \circ g$ is also an
  admissible perturbation. If furthermore 
  $f'(z_0)g'(w_0) \ne 0$ and one of the families is
  essentially quadratic and the other one is sub-quadratic, then the
  composition $h$ is essentially quadratic.
\end{proposition}
\begin{figure}
  \begin{tikzcd}
    D_r(z_0) \arrow[r,"f_a"] \arrow[d,"T_{a,z_0}"] &
    D_s(w_0) \arrow[r,"g_a"] \arrow[d,"T_{a,w_0}"] &
    D_t(g(w_0)) \arrow[d,"T_{a,g(w_0)}"]
    \\
    D_\rho(0) \arrow[r,"F_b"] &
    D_\sigma(0) \arrow[r,"G_b"] &
    D_\tau(0)
  \end{tikzcd}
  \caption{Relation between perturbations $f_a(z) = f(a,z)$ and
    $g_a(w) = g(a,w)$ and their flipped versions $F_b(\zeta) =
    F(b,\zeta)$ and $G_b(\omega) = G(b,\omega)$. The vertical maps are
    maps of the form $T_{a,z_0}(z) = a(z-z_0)$, and the perturbation
    parameters $a$ and $b$ are related by $ab=1$. It is clear from the
    diagram that the flipped perturbation of the composition is the
    composition of the flipped perturbations.}
  \label{fig:flipped-diagram}
\end{figure}

\begin{proof}
  The claim can be proved by a straightforward, though somewhat
  cumbersome calculation of the degrees of the coefficients in the
  composition of the power series. However, we will give a slightly
  more insightful proof here using the flipped perturbations.

  Let as before
  \begin{align*}
    F(b,\zeta) & = \frac{f(b^{-1},z_0 + b \zeta) - f(z_0)}{b} \quad    \text{ and} \\
    G(b,\omega) &= \frac{g(b^{-1}, w_0 + b \omega) - g(w_0)}{b}
  \end{align*}
  be the flipped perturbations associated to $f(a,z)$ and $g(a,w)$,
  and let $h(a,z) = (g \circ f)(a,z) = g(a,f(a,z))$ be their
  composition. Then the flipped perturbation $H(a,z)$ of the
  composition $g \circ f$ is the composition of the flipped
  perturbations $H = G \circ F$, as illustrated in the commutative
  diagram in \autoref{fig:flipped-diagram}.
  If $f$ and $g$ are both admissible, then $F$ and $G$ extend to a
  polydisk centered at $0$, and so does their composition $H$, which
  shows that $g \circ f$ is admissible, too.  If furthermore one of
  $f$ and $g$ is essentially quadratic and the other one is
  subquadratic, then one of their perturbation limits $F$ and $G$ is a
  quadratic polynomial, whereas the other one is linear or
  constant. However, with the additional assumption that
  $f'(z_0)g'(w_0) \ne 0$, neither $F$ nor $G$ can be constant since
  $F'(0) = f'(z_0)$ and $G'(0) = g'(w_0)$.  This shows that the
  perturbation limit of the composition $H = G \circ F$ is the
  composition of a linear function and a quadratic polynomial, so it
  is again a quadratic polynomial, which shows that $h$ is an
  essentially quadratic perturbation.
\end{proof}

\begin{definition}
  A family of maps $f(a,z) = \lambda z + \sum_{k=2}^\infty f_k(a) z^k$
  is \emph{uniformly linearizable} for $|a|\leq r$ if there exists
  $\eps>0$ and a family of conformal maps $h(a,z) = z + O(z^2)$ such
  that $f(a,h(a,z)) = h(a,\lambda z)$ for $|a|\leq r$ and $|z|<\eps$.
\end{definition}

\begin{remark}
  Uniform linearizability means that all maps $z \mapsto f_a(z)$ for
  $|a|<r$ have rotation domains whose size is uniformly bounded
  below. Here ``size'' can be interpreted either as the conformal
  radius or the in-radius of the domain.\footnote{The \emph{conformal
      radius} of a simply connected domain $S \subsetneq \C$ with
    respect to $0 \in S$ is defined as $r_1(S) = h'(0)$ where $h: \D
    \to S$ is the Riemann map with $h(0) = 0$, $h'(0) > 0$. The
    \emph{in-radius} of $S$ with respect to $0$ is $r_2(S) =
    \dist(0,\partial S)$. By the Schwarz lemma and Koebe's distortion
    theorems, $1 \le r_1(S)/r_2(S) \le 4$.}
\end{remark}

The next proposition is the main result of this section, and it
generalizes similar results by P{\'e}rez-Marco (see
\cite{PerezMarco1997} or \cite{Geyer1998}) and Yoccoz
\cite{Yoccoz1995}. P{\'e}rez-Marco used Hartogs' Theorem in his proof,
here we give a proof using the formal linearization following Yoccoz.
\begin{proposition}
  \label{prop:maxprinc}
  If an essentially quadratic family $f(a,z) = \lambda z +
  \sum\limits_{k=2}^\infty f_k(a) z^k$ is uniformly linearizable for $|a|\leq
  r$ then the quadratic polynomial $F(z) = \lambda z + z^2$ is
  linearizable.
\end{proposition}
\begin{proof}
  The idea of the proof is that $f(a,z)$ for large $|a|$ is conjugate
  to the quadratic polynomial, for small $|a|$ it is uniformly
  linearizable, and there is a ``maximum principle'' for linearization
  (in terms of $b=a^{-1}$) which yields linearizability of the
  quadratic polynomial. In order to make sense of perturbations for
  large $a$, we pass to the flipped family
  \[
    F(b,w) = b^{-1} f(b^{-1}, bw) = \lambda w + \sum_{k=2}^\infty
    F_k(w) w^k
  \]
  and use \Cref{prop:flipped-perturbations} to see that it extends
  analytically to a polydisk $|b| \le 1/r$, $|w| \le \delta$, for some
  $\delta > 0$, with perturbation limit $F(0,w) = \lambda w + c w^2$
  with $c = f_2(0) \ne 0$. Using the linear conjugation $z \mapsto cz$
  for the original function $f$ and the perturbation family, we may
  assume that $c = 1$.

  If we try to linearize $F(b,w)$ by a formal power series $H(b,z) = z
  + \sum\limits_{n=2}^\infty H_n(b) z^n$, i.e., solving the equation
  \[
    F(b,H(b,z)) = H(b,\lambda z),
  \]
  we get recursive equations for the coefficients $H_n(b)$ of the form
  \[
    H_n = \frac{F_n + P_n(F_2, \ldots, F_{n-1}, H_2, \ldots,
      H_{n-1})}{\lambda^n - \lambda},
  \]
  where the $P_n$ are explicitly calculable polynomials. In
  particular, since the coefficients $F_n(b)$ are polynomials, we get
  by induction that $H_n(b)$ are polynomials in $b$, too. The series
  $H(b,z) = z+ \sum_{n=2}^\infty H_n(b) z^n$ is the unique normalized
  formal linearizing series for $F(b,w)$, and the function $F(b,w)$ is
  linearizable for a particular $b$ iff the series for $H(b,z)$ has a
  positive radius of convergence.
  
  For $|b| = 1/r$ we know that $H(b,z)$ actually converges in some
  disc $|z| < \delta$, because $F(b,\cdot)$ is linearly conjugate to
  $f(b^{-1},\cdot)$ by $z \mapsto bz$. Furthermore, for fixed $b$ with
  $|b|=1/r$, the map $z\mapsto H(b,z)$ is a normalized conformal map
  in $D_\delta(0)$, thus we get $|H_k(b)| \leq k \delta^{-k+1}$ by de
  Branges's Theorem.  (We do not really need this strong result, the
  classical estimates derived from Cauchy's formula and Koebe's
  distortion theorems would suffice here.)  The maximum principle then
  yields $|H_k(0)| \leq k \delta^{-k+1}$ which implies that $H(0,z)$
  converges for $|z|<\delta$. As $F(0,w) = \lambda w + w^2 = F(w)$, we
  have shown that $F(w)$ is linearizable.
\end{proof}

Combining this result with Yoccoz's result about the optimality of the
Brjuno condition for the quadratic family (Theorem \ref{thm:RBY}), we
immediately get the following result.

\begin{corollary}
  \label{cor:EssentialQuadraticLinearizableFamily}
  If an analytic linearizable germ $f(z) = \lambda z + O(z^2)$ with $\lambda =
  e^{2\pi i \alpha}$ admits an essentially quadratic uniformly
  linearizable perturbation, then $\alpha \in \cB$.
\end{corollary}

\section{Analytic Families of Polynomial-like Maps}
\label{sec:analyt-famil-polyn}
Polynomial-like maps and analytic families of polynomial-like maps were
introduced by Douady and Hubbard in \cite{DouadyHubbard1985}. In this
section we review some of the definitions and results we are going to
use, as well as prove a few small results of our own. We assume that
the reader is familiar with the basics of quasiconformal maps, see
e.g.\ \cite{LehtoVirtanen1973}. For a good overview of various
applications of quasiconformal maps in complex dynamics, see also
\cite{BrannerFagella2014}.

\subsection{Polynomial-like Maps}
\begin{definition}
  A \emph{polynomial-like map} of degree $d \ge 2$ is a triple
  $(f,U,V)$, where $U,V \subset \C$ are bounded simply connected domains
  with $\overline{U} \subset V$, and $f:U \to V$ is a proper analytic
  map of topological degree $d$. The \emph{filled-in Julia set} of
  $(f,U,V)$, denoted by $K(f,U,V)$, is the set of all $z \in U$
  such that $f^n(z) \in U$ for all $n \ge 1$. The \emph{Julia set} of
  $(f,U,V)$ is defined as $J(f,U,V) = \partial K(f,U,V)$.
\end{definition}
The definition is modeled on the dynamics of polynomials. In
particular, every polynomial $f$ of degree $d \ge 2$ is
polynomial-like of the same degree $d$, with $V = \D_r$ being a large
disk, and $U = f^{-1}(V)$ its preimage. The Julia set and filled-in
Julia set are the same in this case, no matter whether $f$ is viewed
as a polynomial, or $(f,U,V)$ as a polynomial-like map.

In the context of polynomial-like maps $(f,U,V)$, we treat $f(z)$ as
undefined whenever $z \notin U$. E.g., when talking about a periodic
point $f^q(z) = z$, it is understood that $f^k(z) \in U$ for $k = 0,
1, \ldots, q-1$, even if $f$ is the restriction of a map defined in a
larger domain. Similarly, the preimage $f^{-1}(W)$ is defined as the
set of all $z \in U$ such that $f(z) \in W$.

\begin{definition}
  Two polynomial-like maps $(f_1,U_1,V_1)$ and $(f_2,U_2,V_2)$ with
  filled-in Julia sets $K_1$ and $K_2$ are \emph{topologically
    conjugate} if there exists a homeomorphism $\phi$ from a
  neighborhood of $K_1$ onto a neighborhood of $K_2$ such that $\phi
  \circ f_1 = f_2 \circ \phi$ near $K_1$. The maps are
  \emph{(quasi-)conformally conjugate} if $\phi$ can be chosen to be
  (quasi-)conformal. They are \emph{hybrid conjugate} if $\phi$ can be
  chosen to be quasiconformal with $\bar{\partial} \phi = 0$ a.e.\ on
  $K_1$.
\end{definition}

Given a polynomial-like map $(f,U,V)$, and a simply-connected domain
$V' \subset V$, let $\gamma$ be a simple closed loop in its preimage
$U' = f^{-1}(V')$, bounding a Jordan domain $W$. By continuity and
properness of the map $f$ we know that
$\partial f(W) \subseteq f(\gamma)$, so $f(W)$ is bounded by a compact
subset of $V'$, which means that $f(W) \subseteq V'$, so that
$W \subseteq U'$. This shows that $\gamma$ is null-homotopic in $U'$,
and thus that every connected component of $U'$ is simply connected.
If $U_1$ is a connected component of $U'$, then the Riemann-Hurwitz
formula gives that the number $n_1$ of critical points of $f$ in $U_1$
and the degree $d_1$ of $f|_{U_1}$ are related by $n_1 = d_1-1$. This
shows that the number of connected components $m$ of $U'$, the total
number $n$ of critical points of $f$ in $U'$ and the degree $d$ of
$f|_{U'}$ are related by $m+n = d$.  If we assume that $V'$ contains
all the critical values of $f$, then $m=1$, i.e., $U'$ is connected
and simply connected.  In particular, if $V'$ is bounded by an
analytic curve in $V$ which is sufficiently close to $\partial V$,
then $U'=f^{-1}(V')$ is an analytic Jordan domain, and $(f,U',V')$ is
polynomial-like and conformally conjugate to $(f,U,V)$ (via the
identity map near $K_f$.)  This shows that up to conformal conjugacy
we can always assume that the domains $U$ and $V$ are analytic Jordan
domains, and that $f$ extends analytically to a neighborhood of
$\overline{U}$.

The most important general result about polynomial-like maps is the
following Straightening Theorem by Douady and Hubbard
\cite{DouadyHubbard1985}.
\begin{theorem}[Douady, Hubbard]
  \label{thm:Straightening}
  Every polynomial-like map $(f,U,V)$ is hybrid conjugate to a
  polynomial $P$. If the filled-in Julia set $K(f,U,V)$ is connected,
  then $P$ is unique up to affine conjugation.
\end{theorem}

This theorem implies that any result in complex dynamics invariant
under hybrid conjugacy is automatically valid for polynomial-like
maps, too. Many of these results can actually be proven directly
without resorting to the Straightening Theorem, by copying the proofs
for polynomials.

\subsection{Analytic Families of Polynomial-like Maps}
\begin{definition}
  Let $A$ be a complex manifold, $\cF = \{ (f_a, U_a, V_a) : a \in A
  \}$ be a family of polynomial-like maps, $\cU = \{ (a,z) : a \in A,
  \, z \in U_a \}$, $\cV = \{(a,z): a \in A, z \in V_a \}$. Then $\cF$
  is an \emph{analytic family of polynomial-like maps} if
  \begin{enumerate}
  \item \label{StandardTopology}
    $\cU$ and $\cV$ are homeomorphic over $A$ to $A \times \D$.
  \item \label{ProperParameters}
    The projection from the closure of $\cU$ in $\cV$ to $A$ is
    proper.
  \item \label{ProperMap}
    The mapping $F:\cU \to \cV$, $F(a, z) = (a,f_a(z))$ is
    complex-analytic and proper.
  \end{enumerate}
\end{definition}
Here \emph{homeomorphic over $A$} means that there exists a
homeomorphism of the form $\phi(a,z) = (a,\phi_a(z))$, and
\emph{proper} means that preimages of compact sets are compact. We
will always assume that $A$ is connected which implies that the degree
of the polynomial-like maps in the family is constant.

The next proposition is not explicitly stated in the paper of Douady
and Hubbard, but it is certainly known to the experts. Roughly
speaking it says that small analytic perturbations of polynomial-like
maps form an analytic family of polynomial-like maps.
\begin{proposition}
  \label{prop:PolynomialLikeStability}
  Let $(f,U,V)$ be polynomial-like of degree $d \ge 2$, and let
  $f_a(z) = f(a,z)$ be complex-analytic in $D_r(0) \times U$ for some
  $r>0$, with $f_0 = f$. Let $K$ be any compact set with $K(f,U,V)
  \subseteq K \subset U$. Then there exists $\rho>0$, a domain $V'$
  and a family of domains $U_a$ for $|a| < \rho$ with $K \subset U_a
  \subset \overline{U_a} \subset U \subset V' \subset V$ such that
  $\cF = \{ (f_a, U_a, V') : |a| < \rho \}$ is an analytic family of
  polynomial-like maps of degree $d$.
\end{proposition}
\begin{proof}
  We may assume that $K$ contains all the (finitely many) critical
  points of $f$ in $U$.  Let $\gamma$ be an analytic Jordan curve in
  $V$ which separates $\partial V$ from $\overline{U} \cup f(K)$, and
  let $V' \subset V$ be the domain bounded by $\gamma$. Then $U_0 =
  f^{-1}(V')$ is a connected and simply connected domain with analytic
  boundary, satisfying $K \subset U_0 \subset \overline{U_0} \subset U
  \subset V'$.  Let $\eta = \eta_0$ be the analytic Jordan curve
  bounding $U_0$. We may pull back the analytic parametrization of
  $\gamma: \R / \Z \to \C$ to obtain an analytic parametrization of
  $\eta_0: \R / \Z \to \C$ with $f(\eta_0(t)) = \gamma(dt)$ for $t\in
  \R / \Z$. By the complex implicit function theorem there exists
  $\rho>0$ and an analytic family of analytic Jordan curves $\eta_a$
  for $|a| < \rho$, satisfying $f_a(\eta_a(t)) = \gamma(dt)$ and
  $\eta_a(t) \in U \setminus K$ for all $t \in \R / \Z$. Let $U_a$ be
  the domain bounded by $\eta_a$. The functional equation implies that
  $f_a$ has degree $d$ on $\partial U_a$ for $|a| < \rho$ so $f:U_a
  \to V'$ is proper of degree $d$ with $K \subset U_a \subset
  \overline{U_a} \subset U \subset V'$, which means that $(f,U_a,V')$
  is polynomial-like of degree $d$. Given any $\eps > 0$, by possibly
  choosing $\rho>0$ smaller, we can make sure that all the curves
  $\eta_a(\R / \Z)$ for $|a| < \rho$ are contained in an
  $\eps$-neighborhood of $\partial U_0$, which implies that $U_a$ is
  contained in an $\eps$-neighborhood of $U_0$. Choosing $\eps =
  \frac12 \dist (\partial U_0, \partial U)$, this yields a compact
  neighborhood $K'$ of $U_0$ such that $U_a \subset K' \subset U$ for
  all $|a|<\rho$.
  
  We have to show that this family satisfies the three properties in
  the definition of analytic families of polynomial-like maps. Note
  that in our case $\cV = D_\rho(0) \times V'$ is a product, and $\cU =
  \{ (a,z): |a| < \rho, \, z \in U_a \} = \{ (a,z): |a|< \rho, \,
  f_a(z) \in V' \} = F^{-1}(\cV) \subseteq D_\rho(0) \times K'$.

  {\sc Property (\ref{ProperMap}).}  By assumption, the map $F(a,z) =
  (a,f(a,z))$ is complex-analytic in $D_r(0) \times U \supset \cU$.
  If $\cK \subset \cV$ is compact, then $\cK \subset K_1 \times K_2$
  with $K_1 \subseteq D_\rho(0)$ and $K_2 \subseteq V'$ compact. By
  continuity, $F^{-1}(\cK)$ is a relatively closed subset of the
  domain $D_\rho(0) \times U$, and additionally $F^{-1}(\cK)$ is
  contained in the compact set $K_1 \times K'$, which itself is
  contained in the domain of $F$, so $F^{-1}(\cK)$ is compact. This
  shows that $f$ is proper.

  {\sc Property (\ref{ProperParameters}).}  Let $\cU'$ be the closure
  of $\cU$ in $\cV$. Clearly, $\cU'$ contains the union of the
  fiberwise closures $\cU'' = \left\{ (a,z): |a|<\rho, \, z \in
    \overline{U_a} \right \}$, and we claim that these two sets $\cU'$
  and $\cU''$ are actually equal. In order to show the opposite
  inclusion, let $(a,z) \in \cU'$. Then $|a|<\rho$, $z \in V'$, and
  there exists $(a_n, z_n) \in \cU'$ with $(a_n, z_n) \to (a,z)$, so
  in particular $z_n \in U_{a_n}$ and $f_{a_n}(z_n) \in V'$. Passing
  to the limit, continuity implies $f_a(z) \in
  \overline{V'}$, and thus $z \in f_a^{-1}(\overline{V'}) =
  \overline{U_a}$, which shows that $(a,z) \in \cU''$.
  
  Now if $K \subset D_\rho(0)$ is compact, then the preimage in $\cU'$
  of $K$ under the projection is $\left\{ (a,z): a \in K , \, z \in
    \overline{U_a} \right\} = \left\{ (a,z): a \in K, \, f_a(z) \in
    \overline{V'} \right\} = F^{-1}(K \times \overline{V'}) \subset K
  \times K'$ which by the same argument as above is a relatively
  closed subset of a compact subset of the domain of $F$, so it is
  itself compact. This shows that the projection from the closure of
  $\cU$ in $\cV$ to $D_\rho(0)$ is proper.

  {\sc Property (\ref{StandardTopology}).}  By possibly choosing
  $\rho$ smaller, we may assume that there exists a point $z_0$ such
  that $z_0 \in U_a$ for all $|a| < \rho$.  Let $\phi_a: \D \to U_a$
  be the conformal map with $\phi(0) = z_0$ and $\phi'(0) > 0$. Since
  the boundaries of $U_a$ move analytically, Carath\'eodory's kernel
  convergence theorem shows that $a \mapsto \phi_a$ is continuous for
  $|a| < \rho$, with respect to the topology of locally uniform
  convergence of analytic functions on $\D$. This implies that
  $\phi(a,z) =(a, \phi_a(z))$ is a continuous bijective map from
  $D_\rho(0) \times \D$ to $\cU$, and that $\cU$ is an open subset of
  $\C^2$. Since both the domain and range are open subsets of $\C^2$,
  the map $\phi$ is a homeomorphism between them by Brouwer's
  invariance of domain. For the image domain $\cV = D_\rho(0) \times
  V'$ the corresponding argument is simpler. Let $\psi: \D \to V'$ be
  a conformal map. Then $\psi(a,z) = (a, \psi(z))$ is a homeomorphism
  from $D_\rho(0) \times \D$ onto $\cV$.

\end{proof}

\section{\texorpdfstring{$J$-stability}{J-stability}}
\label{sec:j-stability}
The concept of $J$-stability was introduced in
\cite{ManeSadSullivan1983} for families of rational functions. Here we
are using a version of this concept and the main results for analytic
families of polynomial-like mappings, as proved in
\cite{DouadyHubbard1985}. In order to simplify notation, we will work
with a fixed analytic family of polynomial-like maps $\cF = \{ (f_a,
U_a, V_a) : a \in A \}$, and write $J_a$, $K_a$ for the Julia set and
filled-in Julia set of $(f_a, U_a, V_a)$, resp.

\begin{definition}
  An indifferent periodic point $z_0$ of $f_{a_0}$ with minimal period
  $n$ is called \emph{persistent} if there exist neighborhoods $B$ and
  $W$ of $a_0$ and $z_0$, respectively, such that for all $a \in B$,
  the map $f_a$ has exactly one periodic point $z(a)$ of minimal
  period $n$, and such that $|(f_a^n)'(z(a))| = 1$ for $a \in B$.  Let
  $S = S(\cF) \subseteq A$ be the interior of the set of parameters
  $a \in A$ for which all indifferent periodic points of $f_a$ are
  persistent. We call $S$ the set of \emph{$J$-stable parameters} in
  the family $\cF$, and we say that a map $f_a$ is \emph{$J$-stable
    (in the family $\cF$)} if $a \in S$.
\end{definition}
\begin{remark}
  Note that in particular any open set subset of $A$ for which the
  corresponding maps $f_a$ have no indifferent periodic points at all
  is a subset of the set of $J$-stable parameters.
\end{remark}
The main result and justification for the name ``$J$-stable'' is the
following adaptation of Ma\~n\'e, Sad, and Sullivan's result
\cite{DouadyHubbard1985}*{II.4, Proposition 10}.
\begin{proposition}[Douady, Hubbard] \label{prop:JStability}
  The set $S$ is open and dense in $A$. Furthermore, for any $a_0 \in
  S$ there exists $K \ge 1$, a neighborhood $B$ of $a_0$ in $S$, a
  neighborhood $W$ of $J_{a_0}$, and a continuous embedding $\phi: B
  \times W \to \cV$ of the form $\phi(a,z) = (a, \phi_a(z))$ such that
  \begin{enumerate}
  \item $a \mapsto \phi_a(z)$ is holomorphic for every $z \in W$.
  \item $z \mapsto \phi_a(z)$ extends to a $K$-quasiconformal map of
    the plane for every $a \in B$.
  \item The image of $\phi$ is a neighborhood of $\cJ_B = \{ (a,z):
    a\in B, z \in J_a \}$ which is closed in $\cV \cap (B \times
    \C)$.
  \item $\phi_{a_0}(z) = z$ for all $z\in W$.
  \item \label{Conjugacy} $\phi_a(J_{a_0}) = J_{a}$, and $f_a \circ
    \phi_a = \phi_a \circ f_{a_0}$ on $J_{a_0}$, for all $a \in B$.
  \end{enumerate}
\end{proposition}
\begin{remark}
  As was pointed out by the referee, the definition of persistent
  indifferent periodic points in Douady and Hubbard
  \cite{DouadyHubbard1985} is slightly different and not equivalent to
  the definition in Ma\~n\'e, Sad, and Sullivan
  \cite{ManeSadSullivan1983}, and the conjugacy relation
  (\ref{Conjugacy}) is stated in a different form in
  \cite{DouadyHubbard1985}. See \Cref{app:stable-persistent} for
  details and discussion.
\end{remark}
We are mostly interested in the following corollary on the persistence
of Siegel disks and Cremer points on the $J$-stable set.
\begin{corollary}
  \label{cor:JStableUniformlyLin}
  Let $a_0 \in S$, and assume that $z_{a_0} \in U_0$ is an
  irrationally indifferent periodic point of $f_{a_0}$ of period $q
  \ge 1$ and multiplier $\lambda = (f_{a_0}^q)'(z_{a_0})$. Then there
  exists a neighborhood $B$ of $a_0$ in $S$ and an analytic map
  $a\mapsto z_a$ in $B$ such that $z_a$ is an irrationally indifferent
  periodic point of $f_a$ of period $q$ and multiplier
  $\lambda$. Furthermore, if $z_{a_0}$ is a Cremer point of $f_{a_0}$,
  then $z_a$ is a Cremer point of $f_a$ for $a \in B$, and if
  $z_{a_0}$ is a Siegel point for $f_{a_0}$, then the family of maps
  $g_a(z) = f_a^q(z+z_a) - z_a$ is uniformly linearizable for $a \in
  B$.
\end{corollary}
\begin{proof}
  By the implicit function theorem, there is a connected neighborhood
  $B$ of $a_0$ in $S$, a neighborhood $W$ of $z_{a_0}$ in $\C$, and an
  analytic map $a \mapsto z_a$ such that $f_a^q(z_a) = z_a$, and such
  that $f_a$ does not have any other $q$-periodic point in $W$. The
  multiplier $\lambda_a = (f_a^q)'(z_a)$ is an analytic function of
  $a$. By definition all indifferent periodic points of $f_a$ in $U_a$
  are persistent throughout the $J$-stable parameter set $S$, so
  $|\lambda_a| = 1$ for all $a \in B$. By analyticity this implies
  that $\lambda_a = \lambda_{a_0} = \lambda$ for all $a \in B$.

  If $z_{a_0}$ is non-linearizable for $f_{a_0}$, then
  $z_{a_0} \in J_{a_0}$, so $z_a = \phi_a(z_{a_0}) \in J_a$ is
  non-linearizable for $f_a$. If $z_{a_0}$ is linearizable for
  $f_{a_0}$, then $z_{a_0} \notin J_{a_0}$, so
  $r = \dist (z_{a_0}, J_{a_0}) > 0$.  By (\ref{Conjugacy}) of
  \autoref{prop:JStability} we know $J_a = \phi_a(J_{a_0})$, and this
  immediately implies that $a \mapsto J_a$ is continuous with respect
  to the Hausdorff metric on the set $S$. By possibly choosing $B$
  smaller, we can make sure that $\dist(z_a, J_a) \ge r/2 > 0$, so
  $z_a \notin J_a$. This implies that $z_a$ is linearizable for $f_a$,
  and that the Siegel disk centered at $z_a$ has in-radius $\ge r/2$,
  so that its conformal radius is also $\ge r/2$. (In fact, this
  argument shows that both the set of linearizable and the set of
  non-linearizable parameters are open subsets of $S$, so
  linearizability or non-linearizability persists across the connected
  component of $S$ containing $a_0$. Uniform linearizability will at
  least hold on compact subsets of stable components.)
\end{proof}

\section{Fatou-Shishikura Inequalities}
\label{sec:fatou-shish-ineq}
Saturated polynomials are those for which the Fatou-Shishikura
inequality on the number of non-repelling cycles is an equality. Using
the standard Fatou-Shishikura inequality that a polynomial of degree
$d \ge 2$ has at most $d-1$ non-repelling cycles, this would be
equivalent to having exactly $d-1$ non-repelling cycles. In order to
get a stronger result, we are proving a stronger version of the
Fatou-Shishikura inequality, taking into account critical relations.

Using and refining Goldberg and Milnor's fixed point portraits from
\cite{GoldbergMilnor1993}, Kiwi proved the following result in
\cite{Kiwi2000}*{Corollary 3.4}.
\begin{theorem}[Kiwi]
  \label{thm:KiwiSeparation}
  Let $P$ be a polynomial of degree $d \ge 2$ with connected Julia set
  $J(P)$. Then
  \begin{enumerate}
  \item \label{SeparationCremer}
    Given a Cremer cycle $Z$, there exists a critical point $c \in
    J(P)$ such that $Z \subseteq \omega_P(c)$ and such that
    $\omega_P(c)$ contains neither any other Cremer point nor any
    non-preperiodic boundary point of a Siegel disk.
    \item \label{SeparationSiegel}
      Given a cycle of Siegel disks $S$, and a point $z \in \partial
      S$, there exists a critical point $c \in J(P)$ such that $z \in
      \omega_P(c)$ and such that $\omega_P(c)$ contains neither any
      other Cremer point nor any non-preperiodic boundary point of a
      Siegel disk.
  \end{enumerate}
\end{theorem}
Note that there are only countably many preperiodic points, so there
always exist non-preperiodic boundary points of Siegel disks. Given an
irrationally indifferent cycle $Z$ of $P$, we will call any critical
point satisfying (\ref{SeparationCremer}) in the Cremer point case or
(\ref{SeparationSiegel}) in the case of a Siegel disk $S$, for any
non-preperiodic $z \in \partial S$, \emph{associated} to the cycle
$Z$. Note that associated critical points for different cycles have
disjoint infinite orbits in the Julia set.

In the following, $P$ is a fixed polynomial of degree $d \ge 2$, not
necessarily with connected Julia set. Given two points
$z_1, z_2 \in \C$, note that their forward orbits $O^+_P(z_1)$ and
$O^+(z_2)$ are either disjoint, or otherwise there exist $n,m \ge 0$
with $P^{m+k}(z_1) = P^{n+k}(z_2)$ for all $k \ge 0$. In the latter
case we say that $z_1$ and $z_2$ are \emph{(forward-orbit) equivalent}
and that they have the same \emph{orbit tail}. It is easy to see that
this is an equivalence relation on $\C$, and that equivalent points
$z_1$ and $z_2$ either both have finite or both have infinite
orbits. Furthermore, by complete invariance of the Julia set,
filled-in Julia set, and basin of infinity, we can talk about
equivalence classes being contained in the Julia set, filled-in Julia
set, or basin of infinity. We are particularly interested in this
equivalence relation restricted to the critical points of $P$.

\begin{definition}
  \label{def:critical-orbits}
  A \emph{critical orbit tail} is the intersection of the forward
  orbits of an equivalence class of critical points. Let
  $n_{\infty,F}(P)$ and $n_{\infty,J}(P)$ denote the number of
  infinite critical orbit tails contained in the Fatou and Julia set
  of $P$, respectively, and let
  $n_\infty(P) = n_{\infty,F}(P) + n_{\infty,J}(P)$ denote the total
  number of infinite critical orbit tails of $P$.
\end{definition}
Note that since $P$ has only finitely many critical points, the
intersection of forward orbits of an equivalence class is non-empty
and contains an actual orbit tail for each critical point in this
equivalence class. In particular, it is finite if and only if every
critical point in the equivalence has a finite forward orbit, and it
is infinite if and only if every critical point in the equivalence has
an infinite forward orbit.

\begin{definition}
  The \emph{weight} of a non-repelling cycle $Z$ is
  \[
\gamma(Z) = 
\begin{cases}
  0 & \text{ if $Z$ is super-attracting} \\
  1 & \text{ if $Z$ is attracting or irrationally indifferent} \\
  r & \text{ if $Z$ is a parabolic cycle with $r$ invariant cyles
    of petals}
\end{cases}
\]
We define $\gamma_{irr}(P)$ as the sum of the weights of all
irrational cycles, $\gamma_{ap}(P)$ as the sum of the weights of all
attracting and parabolic cycles, and $\gamma(P) = \gamma_{irr}(Z) +
\gamma_{ap}(Z)$ as the sum of the weights of all non-repelling cycles
of $P$.
\end{definition}
Note that $\gamma_{irr}(P)$ equals the number of irrationally
indifferent cycles, since each one of them has weight 1.  We will
derive our version of the Fatou-Shishikura inequality from the
following result, which is basically due to Kiwi.
\begin{theorem}
  \label{thm:KiwiFatouShishikura}
  $\gamma_{irr}(P)
  \le n_{\infty, J}(P)$.
\end{theorem}
\begin{proof}
  In the case where the Julia set $J(P)$ is connected, this is an
  immediate consequence of Kiwi's result,
  \autoref{thm:KiwiSeparation}. For every Cremer cycle and every
  Siegel cycle there is at least one associated critical point with
  infinite orbit in $J(P)$ which is not associated to any other Cremer
  or Siegel cycle. This implies that critical points associated to
  different Cremer or Siegel cycles are not equivalent, which shows
  $\gamma_{irr}(P) \le n_{\infty,J}(P)$.

  In the case of disconnected Julia set, we can decompose the dynamics
  of $P$ into a finite number of polynomials with connected Julia sets
  as follows.

  Every non-repelling cycle $Z=(z_1, \ldots, z_q)$ of $P$ is contained
  in some cycle $K(Z) = (K_1, \ldots, K_n)$ of periodic components of
  the filled-in Julia set $K(P)$. The period $n$ of $K(Z)$ always
  divides the period $q$ of $Z$, but it might be strictly
  smaller. Different periodic cycles either correspond to the same or
  to disjoint cycles of components. In this way we obtain a finite
  number of periodic cycles of components of the filled-in Julia
  set. 

  Let $(K_1, \ldots, K_n)$ be such a periodic cycle of components of
  $K(P)$, and let $J_k = \partial K_k$. Let
  $G(z) = \lim\limits_{n\to\infty} d^{-n} \log^+ P^n(z)$ be the
  associated Green's function for $P$. For $\eps > 0$ we define
  $K_1^\eps$ to be the connected component of the sub-level set
  $\{ G < \eps \} = \{ z \in \C : G(z) < \eps \}$ containing $K_1$.
  Since $\bigcap_{\eps > 0} K_1^\eps$ is a connected subset of $K(P)$
  containing $K_1$, it has to be equal to $K_1$. This means that we
  can choose $\eps > 0$ small enough such that $K_1^\eps$ does not
  contain any critical values of $P^n$ in $A_\infty(P)$, and that it
  is disjoint from $P^{-n}(K_1) \setminus K_1$, i.e., that it does not
  intersect any of the other preimages of $K_1$ under $P^n$. With this
  choice of $\eps$, define $V = K_1^\eps$ and $U$ to be the component
  of $P^{-n}(V)$ containing $K_1$ (which is also a component of
  $K_1^{\eps/d^n}$.) Then $U$ and $V$ are simply connected domains
  with $\overline{U} \subseteq V$, and $P^n$ is a proper analytic map
  from $U$ to $V$. Furthermore, since $P^n$ has an indifferent fixed
  point at $z_1 \in K_1 \subset U$, the Schwarz lemma shows that $P^n$
  cannot be a conformal map from $U$ to $V$, so $(P^n, U, V)$ is a
  polynomial-like map of some degree $d_1 \ge 2$. It also shows that
  $P^n$ has at least one critical point in $U$, which by the choice of
  $\eps$ above has to be in $K_1$, showing that $K_1$ is a continuum,
  not just a single point.  Furthermore,
  $P^n(K_1) = K_1 = P^{-n}(K_1) \cap U$, which shows that $K_1$ is a
  completely invariant compact subset of $U$. Since the complement of
  $K_1$ is connected, this shows that $K_1 = K(P^n, U, V)$ is the
  filled-in Julia set of $(P^n, U, V)$.

  By the Straightening Theorem, $(P^n,U,V)$ is hybrid conjugate to a
  polynomial $P_1$ with connected Julia set. Hybrid conjugacies
  preserve critical points and by \cite{PerezMarco1997} they also
  preserve multipliers of irrationally indifferent cycles, so every
  irrationally indifferent cycle of $P$ in $K = K_1 \cup \ldots K_n$
  corresponds to an irrationally indifferent cycle of $P_1$. Applying
  Kiwi's result (\autoref{thm:KiwiSeparation}) to $P_1$, we see that
  $P_1$ has at least one associated infinite critical orbit in
  $J(P_1)$ for every irrationally indifferent cycle, so $P^n$ has at
  least one associated infinite critical orbit in
  $J(P^n,U,V) \subset J_1 \subseteq J(P)$, corresponding to at least
  one infinite critical orbit of $P$ in
  $J_1 \cup \ldots \cup J_n = J(P) \cap (K_1 \cup \ldots \cup K_n)$.
  Since the cycles of filled-in Julia components are either disjoint
  or identical for different cycles, this shows that associated
  critical points of $P$ for different Cremer or Siegel cycles are not
  equivalent, finishing the proof.
\end{proof}

Our version of the Fatou-Shishikura inequality is the following.  The
main difference to the standard statement is that we replace the count
of critical points by the count of infinite critical orbit
tails. I.e., we do not count strictly preperiodic critical points at
all, and we do not double-count multiple critical points or critical
points whose forward orbits eventually collide.
\begin{theorem}
\label{thm:FatouShishikura}
  $\gamma(P) \le n_\infty(P)$.
\end{theorem}
\begin{proof}
  Every rationally indifferent cycle of weight $r$ has $r$ invariant
  cycles of petals attached, and each of them contains at least one
  critical point with infinite forward orbit. Every attracting, but
  not super-attracting cycle contains at least one critical point with
  infinite forward orbit in its attracting cycle of Fatou
  domains. Critical points in disjoint cycles of Fatou domains cannot
  be equivalent, so this shows $\gamma_{ap}(P) \le
  n_{\infty,F}(P)$. From \autoref{thm:KiwiFatouShishikura} we get that
  $\gamma_{irr}(P) \le n_{\infty,J}(P)$. Adding up these inequalities
  we get $\gamma(P) \le n_\infty(P)$.
\end{proof}

One immediate consequence is the following more conventionally stated
version of the Fatou-Shishikura inequality.
\begin{corollary}
  The number of non-repelling cycles of a polynomial of degree $d \ge
  2$ is bounded by the number of critical orbit tails.
\end{corollary}
\begin{proof}
  Let $P$ be a polynomial of degree $d \ge 2$ with $\gamma_0$
  super-attracting cycles of $f$, and $\gamma_1$ non-repelling cycles
  which are not super-attracting. Then $\gamma_1 \le \gamma(P) \le
  n_\infty(P)$ (by definition of $\gamma$ and
  \autoref{thm:FatouShishikura}), so the number of non-repelling
  cycles satisfies $\gamma_0 + \gamma_1 \le \gamma_0 +
  n_\infty(P)$. Distinct super-attracting cycles are non-equivalent
  finite critical orbits, so $f$ has at least $\gamma_0 + n_\infty(P)$
  distinct critical orbit tails.
\end{proof}

By the Straightening Theorem, the following generalization of the
results in this section to polynomial-like maps is immediate.
\begin{corollary}
  \label{cor:FatouShishikuraPolynomialLike}
  Let $(f,U,V)$ be a polynomial-like map of degree $d\ge 2$. Then
  $\gamma_{irr}(f,U,V) \le n_{\infty,J}(f,U,V)$ and
  $\gamma(f,U,V) \le n_\infty(f,U,V)$.
\end{corollary}
The notation here is the obvious generalization of the notation for
polynomials. One little subtlety in this statement is that eventually
undefined critical orbits should be counted as infinite critical
orbits in the Fatou set, since after hybrid conjugacy they will be in
the basin of $\infty$. However, the proofs given above using the
decomposition into polynomial-like maps with connected Julia sets show
that we might as well discard these orbits completely and only count
critical orbits in the filled-in Julia set.

\section{Saturated Polynomials}
\label{sec:satur-polyn}
In this final section we are going to focus on polynomials for which
equality in one of our versions of the Fatou-Shishikura inequality
(\autoref{thm:KiwiFatouShishikura} or \autoref{thm:FatouShishikura})
holds and show that these polynomials do not have exotic Siegel disks.
\begin{definition}
  Let $P$ be a polynomial of degree $d \ge 2$. We say that $P$ is
  \emph{saturated} if $\gamma(P) = n_\infty(P)$, and we say that it is
  \emph{Julia-saturated} if $\gamma_{irr}(P) = n_{\infty,J}(P)$.
\end{definition}
By the discussion of the Fatou-Shishikura inequalities above, it is
clear that saturated polynomials have connected Julia sets.  It is
also easy to see that every saturated polynomial is Julia-saturated,
but the converse is obviously not true, as shown by polynomials with
disconnected Julia sets or polynomials with attracting and/or
parabolic domains which contain several critical orbits.

Intuitively, being saturated means that every super-attracting,
attracting, and irrationally indifferent cycle, as well as every
invariant cycle of petals, has exactly one associated infinite
critical orbit tail, and that all other critical orbits are strictly
preperiodic.

Another way to look at this condition is that a general polynomial
satisfies $\gamma(P)$ algebraic multiplier conditions (including
multiplicity conditions at parabolic points) and $d-1-n_\infty(P)$
``independent'' critical relations which adds up to
$d-1-(n_\infty(P)-\gamma(P))$ algebraic equations. If the associated
varieties in the $(d-1)$-dimensional parameter space of (normalized)
polynomials of the same degree intersect properly, then they should
determine an algebraic set of dimension $n_\infty(P)-\gamma(P)$. In
this point of view, being saturated means that the corresponding
algebraic set is finite, i.e., that $P$ is determined up to finite
ambiguity by its algebraic multiplier conditions and critical
relations. It turns out that this can be made precise in an algebraic
geometric way, but since we do not need it in our proof, we will not
go into details here.

We will first use this algebraic rigidity of saturated polynomials to
show that they do not have exotic Siegel disks, and then show that
Julia-saturated polynomials can be ``made saturated'', so that they do
not have exotic Siegel disks either.

\begin{proposition}
  \label{prop:perturb}
  Let $P$ be a saturated polynomial of degree $d \ge 2$.  Then there
  exists $\rho>0$, a $J$-stable analytic family of polynomial-like maps
  $(P_a, U_a, V)$ of degree $d$ with $K(P) \subset U_a \subset V$ for
  $|a| < \rho$ such that $P_0 = P$ and for every irrationally
  indifferent cycle $Z = (z_1,\ldots,z_q)$ the family $P_a$ is a
  quadratic perturbation at $z_1$ and a sub-quadratic
  perturbation at $z_2,\ldots,z_q$.
\end{proposition}
\begin{proof}
  We are going to build the perturbation in such a way that all
  multipliers, including multiplicities at parabolic points, and all
  critical relations are preserved.

  Let $T$ be an integer large enough so that all critical relations
  are ``observable'' by time $T$, i.e., such that for all equivalent
  critical points $c_1$ and $c_2$ there exist $m,n \le T$ with
  $P^m(c_1) = P^n(c_2)$, and for all critical points $c$ with finite
  orbits there exist $m<n\le T$ with $P^m(c) = P^n(c)$. Choose another
  integer $N>d$ larger than all tangency indices of parabolic periodic
  cycles.

  Let $B$ be the finite set of all non-repelling periodic points of
  $P$, as well as the critical points and their forward orbits up to
  the $T$-th iterate. Let $B_1 \subseteq B$ be a set of
  representatives of irrationally indifferent periodic cycles,
  containing one irrationally indifferent periodic point out of each
  cycle, and let $B_2 = B \setminus B_1$. Define
  \begin{equation*} 
    Q(z) = \prod_{b \in B_1} (z-b)^2 \prod_{b \in B_2} (z-b)^N
    \quad \text{ and } \quad 
    P_a(z) = P(z)+aQ(z).
  \end{equation*} 
  We claim that $P_a$ has the desired properties.
  
  First of all, there exist domains $U$ and $W$ such that $K(P)
  \subset U \subset W$ and such that $(P,U,W)$ is polynomial-like of
  degree $d$. \autoref{prop:PolynomialLikeStability} shows the
  existence of $\rho > 0$, and domains $U_a$ and $V$ with $K(P) \cup B
  \subset U_a \subset V$ such that $(P_a, U_a, V)$ is polynomial-like
  of degreee $d$ for $|a| < \rho$. Whenever $a$ appears in the rest of
  the proof, we will implicitly assume that $|a| < \rho$.
  
  For the rest of the argument, note that $Q$ vanishes to order $N$ at
  all points of $B_2$, so that $P_a^{(k)}(b) = P^{(k)}(b)$ for all
  $|a|<\rho$, $b \in B_2$ and $0 \le k < N$. Furthermore, the chain
  rule for higher derivatives shows that if $b, \, P(b), \ldots,
  P^{n-1}(b) \in B_2$, then $(P_a^n)^{(k)}(b) = (P^n)^{(k)}(b)$ (these
  are the $k$-th derivatives of the $n$-th iterates) for all
  $|a|<\rho$, and $0 \le k < N$.
  
  If $c$ is a critical point of $P$ of multiplicity $m$, then $m+1 \le
  d < N$ and $c \in B_2 \subset U_a$, so $c$ is still a critical point of
  multiplicity $m$ of $(P_a,U_a,V)$. If $c$ has a finite orbit for $P$, then
  there exist $m<n \le T$ such that $P^m(c) = P^n(c)$, and since
  $P^k(c) \in B \subset U_a$ for $0 \le k \le T$, we also get
  that $P_a^m(c) = P_a^n(c)$, so $c$ has a finite orbit for
  $(P_a,U_a,V)$ as well. If two critical points $c_1$ and $c_2$ are
  equivalent for $P$, then there exist $m,n \le T$ such that $P^m(c_1)
  = P^n(c_2)$, and by the same argument as before this implies that
  $P_a^m(c_1) = P_a^n(c_2)$, so they are equivalent for $P_a$, too. In
  particular this argument shows that $n_\infty(P_a,U_a,V) \le
  n_\infty(P)$.
  
  If $Z=(z_1, \ldots, z_q)$ is a rationally indifferent cycle for $P$
  with multiplier $\lambda = e^{2\pi i s/t}$, with $s$, $t$ relatively
  prime integers, $t \ge 1$, and $P^{tq}(z) = z_1 + a_{m+1}
  (z-z_1)^{m+1} + \ldots$, with $a_{m+1} \ne 0$, then we have
  $N>\tau(Z) = m+1$ by definition. Since $Z \subseteq B_2$, we conclude
  that $(P_a)^{tq}(z) = z_1 + a_{m+1}(z-z_1)^{m+1} + \ldots$ as
  well. This shows that the weight of the parabolic cycle is the same
  for $P_a$ as it is for $P$.

  Let $Z=(z_1, \ldots, z_q)$ be an irrationally indifferent cycle of
  $P$. We may assume that $z_1 \in B_1$, and $z_2, \ldots, z_q \in B_2$.
  It is immediate from the definition of $P_a$ that it is a quadratic
  perturbation at $z_1$ and a sub-quadratic perturbation at $z_2,
  \ldots, z_q$. This also implies that $Z$ is again an irrationally
  indifferent cycle for $P_a$ with the same multiplier as for $P$.

  In order to show $J$-stability it is enough to show that for every
  sufficiently small $|a|$ every indifferent periodic point for
  $(P_a, U_a, V)$ is persistent, since the set $S$ of $J$-stable
  parameters is the interior of the set parameters for which all
  indifferent periodic points are persistent. As the previous
  paragraphs show, all indifferent periodic points for $P=P_0$ persist
  for all $a$. Now assume that there exists $a \ne 0$ with a
  non-persistent indifferent periodic point in $U_a$.  Then the
  polynomial-like map $(P_a, U_a, V)$ has
  $n_\infty(P_a,U_a,V) \le n_\infty(P)$ infinite critical orbit tails
  in $U_a$, and since it has at least one additional indifferent
  cycle, we know that $\gamma(P_a,U_a,V) > \gamma(P)$. By assumption
  $P$ is saturated, so $\gamma(P) = n_\infty(P)$, implying that
  $\gamma(P_a,U_a,V) > n_\infty(P_a,U_a,V)$. However, this inequality
  contradicts our version of the Fatou-Shishikura inequality for
  polynomial-like maps, \autoref{cor:FatouShishikuraPolynomialLike}.
\end{proof}

Now we are finally in a position to prove the central result of this
paper.
\begin{theorem}
  \label{thm:SaturatedBrjuno}
  Saturated polynomials do not have exotic Siegel disks.
\end{theorem}
\begin{proof}
  Let $P$ be a saturated polynomial and $Z=(z_1, \ldots, z_q)$ be a
  Siegel cycle of $P$ with multiplier $(P^q)'(z_1) = \lambda = e^{2\pi
    i \alpha}$. Let $(P_a, U_a, V)$ for $|a| < \rho$ be the $J$-stable
  analytic family given by \autoref{prop:perturb}. Since $P_a$ is a
  quadratic perturbation at $z_1$ and a sub-quadratic
  perturbation at $z_2, \ldots, z_q$, the $q$-th iterate
  $P_a^q$ is an essentially quadratic perturbation of $P^q$ at $z_1$
  by \autoref{prop:composition}. Then
  \autoref{cor:JStableUniformlyLin} establishes the existence of $r
  \in (0,\rho)$ such that $g_a(z) = P_a^q(z+z_1) - z_1$ is uniformly
  linearizable for $|a| \le r$. Since $g_a$ is conjugate to $P_a^q$ by
  a simple translation, independent of $a$, it is still uniformly
  linearizable and an essentially quadratic perturbation of $g_0$ at
  $0$, with $g_a'(0)= (P_a^q)'(z_1) = e^{2\pi i \alpha}$. By
  \autoref{cor:EssentialQuadraticLinearizableFamily} this implies
  $\alpha \in \cB$.
\end{proof}

\JuliaSaturatedTheorem
\begin{proof}
  The decomposition technique for the case of disconnected Julia sets
  as employed in the proof of \autoref{thm:KiwiFatouShishikura} shows
  that it is enough to prove the theorem for the case of polynomials
  $P$ with connected Julia set. In \cite{McMullen1988}*{Proposition
    6.9}, McMullen proves the following, using quasiconformal surgery
  (which in that paper is called ``conformal surgery'') to glue
  ``rigid models'' into each Fatou component: There exists a
  polynomial $Q$ and a quasiconformal map $\phi: \C \to \C$ with
  $\phi(J(P))=J(Q)$ and $\phi \circ P = Q \circ \phi$ on $J(P)$, such
  that $Q$ is rigid on the Fatou set in the following sense. Every
  periodic Fatou component is either super-attracting, a Siegel disk,
  or a parabolic petal. All critical points in super-attracting basins
  are periodic, and all critical points in preimages of
  super-attracting basins are preperiodic. All critical points in
  preimages of Siegel disks are preperiodic. Every parabolic basin
  contains exactly one critical orbit tail (which is necessarily
  infinite.)\footnote{Note that this refers to the whole parabolic
    basin of a periodic cycle of petals, not just the petals
    themselves. McMullen constructs $Q$ so that there exists a
    periodic petal $U_0$ with exactly one critical point $c_0$, and
    that any other component of the parabolic basin has at most one
    critical point, and so that all these critical points are mapped
    to $c_0$ by some iterate of $Q$.} The polynomials $P$ and $Q$ have
  the same number of Siegel cycles, Cremer cycles, and invariant
  cycles of petals. The total number of attracting plus
  super-attracting cycles is also the same, only that $Q$ has no
  attracting cycles. Assuming that $P$ is Julia-saturated, we know
  that every infinite critical orbit tail in the Fatou set of $Q$
  corresponds to an invariant cycle of petals. The number of critical
  orbits in the Julia set is the same for $P$ and $Q$, and by
  assumption equals the number of irrationally indifferent cycles of
  $P$. This shows that the total number of infinite critical orbits of
  $Q$ equals the number of irrationally indifferent cycles plus the
  number of invariant cycles of petals. Since $Q$ has no attracting
  periodic points, this shows that $\gamma(Q) = n_\infty(Q)$, so $Q$
  is saturated. By \autoref{thm:SaturatedBrjuno}, every Siegel disk
  for $Q$ has a Brjuno rotation number, and since rotation numbers of
  corresponding Siegel disks of $P$ and $Q$ are the same, this shows
  that all Siegel disks for $P$ have Brjuno rotation numbers. For some
  more details about McMullen's construction and the argument that
  these ``stable conjugacies'' preserve parabolic points, Cremer points,
  Siegel disks, as well as their rotation numbers, see
  \autoref{app:stable-persistent}.
\end{proof}

We conclude this section with an application to certain concrete
families of polynomials.

\MainCorollary

\begin{proof}
  The family $P_{c,d}$ has all critical points at 0, thus it can have
  at most one infinite critical orbit tail. Whenever it has an irrationally
  indifferent periodic point, it is saturated. In the family
  $Q_{c,d}$, for $c=0$ the map $Q_{0,d}(z) = z + z^d$ has all critical
  points in the immediate basins of the parabolic fixed point at zero,
  so there are no Siegel disks at all. For $c \ne 0$, the map
  $Q_{c,d}$ has a fixed point of multiplicity $d-1$ at 0, thus it has
  $d-2$ fixed attracting petals, and $\gamma(0) = d-2$. Whenever there
  is a Siegel cycle, we get $\gamma(P) = d-1$, and so $P$ is
  saturated, too.
\end{proof}

For rational functions the techniques in this paper do not
work. However, using Shishikura's quasiconformal surgery technique
from \cite{Shishikura1987} instead of polynomial-like maps, and using
a rigidity result of McMullen from \cite{McMullen1987}, a similar
result for a more restricted class of rational functions is proved in
\cite{Manlove2015}, and will be pursued in a forthcoming paper
\cite{GeyerManlove2015}.

\appendix
\section{Stable Conjugacies and Persistent Periodic Points}
\label{app:stable-persistent}
Prompted by the observation of the referee that the definitions of
persistent periodic points in \cite{DouadyHubbard1985} and
\cite{ManeSadSullivan1983} are not equivalent, here is a short
overview of the problems and inconsistencies in the literature, as
well as our chosen way of dealing with it. The results from both of
these papers which we are mostly interested in are that for
holomorphic families of maps there is an open and dense set where all
indifferent periodic points are persistent, and that the family is
\emph{stable} (also called \emph{$J$-stable}) on this set. We start by
reviewing the concept of stable conjugacies, based on McMullen's
approach in \cite{McMullen1988}.

\subsection{Stable Conjugacies}
\begin{definition}
  Let $(f_1, U_1, V_1)$ and $(f_2, U_2, V_2)$ be polynomial-like maps
  with Julia sets $J_1$ and $J_2$, respectively. A homeomorphism
  $\phi:J_1 \to J_2$ is a \emph{stable conjugacy} between
  $(f_1, U_1, V_1)$ and $(f_2, U_2, V_2)$ iff
  $\phi \circ f_1 = f_2 \circ \phi$ on $J_1$, and if $\phi$ extends to
  a quasiconformal map $\phi: \C \to \C$.
\end{definition}
\begin{remark}
  Note that it is not assumed that the conjugacy extends to a
  neighborhood of $J_1$, only that $\phi$ as a map extends
  quasiconformally.
\end{remark}
One remarkable fact is that stable conjugacies preserve local dynamics
at indifferent periodic points, in the following sense:
\begin{proposition}
  \label{prop:stable-conj}
  Let $(f_1, U_1, V_1)$ and $(f_2, U_2, V_2)$ be polynomial-like maps
  with Julia sets $J_1$ and $J_2$, respectively, and assume that there
  exists a stable conjugacy $\phi:J_1 \to J_2$ between them. Then the
  following statements hold true.
  \begin{enumerate}
  \item\label{stable-conj-1} If $z_1 \in J_1$ is an indifferent
    periodic point of $f_1$ of period $q$, with multiplier $\lambda$,
    then $z_2 = \phi(z_1) \in J_2$ is an indifferent periodic point of
    $f_2$, again of period $q$ and with multiplier $\lambda$. In the
    case of parabolic periodic points, the tangency index $\tau$ is
    also preserved under stable conjugacy.
  \item\label{stable-conj-2} If $D_1$ is a periodic Siegel disk for
    $f_1$, of period $q$, with associated multiplier $\lambda$, then
    $D_2 = \phi(D_1)$ is a periodic Siegel disk for $f_2$, with the
    same period $q$ and associated multiplier $\lambda$.
  \end{enumerate}
\end{proposition}
\begin{remark}
  This theorem also holds for families of rational maps, with
  essentially the same proof. In that case, one can also prove that
  stable conjugacies preserve Herman rings and their rotation
  numbers.
\end{remark}
\begin{proof}
  By \cite{McMullen1988}, stable conjugacies induce quasisymmetric
  conjugacies between ideal boundaries of Fatou components, and the
  quasisymmetric conjugacy class on the ideal boundary of a fixed
  Fatou component determines whether that component is
  (super-)attracting, parabolic, or a Siegel disk, and in the case of
  a Siegel disk it also determines the rotation number. By passing to
  an iterate, the same is true for periodic Fatou components. In
  particular, this shows that stable conjugacies map Siegel disks to
  Siegel disks with the same period and same rotation number,
  establishing property (\ref{stable-conj-2}). Furthermore, this shows
  that stable conjugacies map any parabolic periodic Fatou component
  of $f_1$ associated to some parabolic periodic point $z_1$ to a
  parabolic periodic Fatou component of $f_2$ associated to the
  parabolic periodic point $z_2 = \phi(z_1)$. This implies that stable
  conjugacies preserve parabolic periodic points, together with the
  number of associated parabolic periodic Fatou components, as well as
  their cyclic ordering, and this shows that stable conjugacies
  preserve periods, multipliers, and tangency indices of parabolic
  periodic points.

  The one property which remains to show is that stable conjugacies
  map Cremer points for $f_1$ to Cremer points of $f_2$ with the same
  period and multiplier. By passing to an iterate, we may assume that
  we have a Cremer fixed point for $f_1$, and by conjugation with
  translations we may assume that $z_1 = 0$ and $z_2 = \phi(z_1) = 0$.
  Since $z_2 \in J_2$, and since stable conjugacies (and their
  inverses) map parabolic points to parabolic points, we know that
  $z_2$ must be either a Cremer point or repelling. If $z_2$ was
  repelling for $f_2$, we would get that there exists $\delta_2 > 0$ and
  $c_2 \in (0,1)$ with $|f_2^{-n}(z)| \le c_2^n$ for $z \in J_2$,
  $|z|<\delta_2$, where $f_2^{-n}$ denotes the local branch of the
  inverse of $f_2^n$ fixing $0$. Since quasiconformal maps are
  H\"older continuous, this implies that there exist $\delta_1 > 0$
  and $c_1 \in (0,1)$ with $|f_1^{-n}(z)| \le c_1^n$ for $z \in J_1$,
  $|z|<\delta_1$, contradicting the assumption that $|f_1'(0)| = 1$.

  The fact that multipliers of Cremer points are invariant under
  stable conjugacies follows from P\'erez-Marco's theory of hedgehogs,
  see \cite{PerezMarco1997}, and also \cite{Risler1999} and
  \cite{Childers2008}. We write $f_k'(0) = e^{2\pi i \alpha_k}$, with
  rotation numbers $\alpha_1,\alpha_2 \in (0,1)$ for $f_1$ and
  $f_2$. Fixing a small disk $D_1$ centered at $0$, there exists a
  locally completely invariant continuum
  $K_1\subseteq \overline{D_1} \cap J_1$ with connected complement
  $\C \setminus K_1$, containing $0$ and at least one point on
  $\partial D_1$. Furthermore, there is a conformal map
  $h: \C\setminus \overline{\D} \to \C \setminus K_1$, and the
  conjugate map $g_1 = h^{-1} \circ f_1 \circ h$ extends to an
  analytic circle diffeomorphism with rotation number $2\pi
  \alpha_1$. Equivalently, the map $f_{\C \setminus K_1}$ induces an
  analytic circle diffeomorphism with rotation number $2\pi \alpha_1$
  on the prime end boundary of $\C \setminus K_1$. Then the image
  $K_2 = \phi(K_1)$ of $K_1$ under the stable conjugacy is a hedgehog
  for $f_2$ in $D_2 = \phi(D_1)$, and $\phi$ conjugates $f_1|_{K_1}$
  to $f_2|_{K_2}$, which implies that $\phi$ induces an
  orientation-preserving conjugacy between the induced maps on prime
  ends of $f_1|_{\C \setminus K_1}$ and $f_2|_{\C \setminus K_2}$. By
  invariance of rotation numbers under orientation-preserving
  conjugacy and by P\'erez-Marco's result that the induced circle
  diffeomorphism on prime ends of a hedgehog has the same rotation
  number as the Cremer fixed point, we get that $\alpha_1 = \alpha_2$,
  and thus $f_1'(0) = f_2'(0)$ as claimed.
\end{proof}
  
\subsection{Persistent Periodic Points and
  \texorpdfstring{$J$-stability}{J-stability}}
The paper \cite{ManeSadSullivan1983} deals with analytic families of
rational maps, whereas \cite{DouadyHubbard1985} concerns analytic
families of polynomial-like maps. For the main ideas of this appendix,
this difference is immaterial, and for simplicity we will illustrate
it first in analytic families of polynomials. Also, since we are
mostly interested in local perturbations, we will assume that our
families are parametrized over the unit disk.

\begin{definition}
  Let $\cF = \{ f_a : a \in \D \}$ be an analytic family of
  polynomials and let $a_0 \in \D$ be a parameter for which
  $f_{a_0}$ has an indifferent periodic point $z_0 = f_{a_0}^n(z_0)$
  of minimal period $n \ge 1$. Then this periodic point is
  \begin{enumerate}
  \item \label{def:mss-persistent} \emph{MSS-persistent} if the
    projection $P_n: M_n \to \D$, $P_n(a,z) = a$, from the set $M_n =
    \{ (a,z) \in \D \times \C : f_a^n(z) = z, \, f_a^k(z) \ne z \text{
      for } 0<k<n \}$ is locally injective near $(a_0, z_0)$, and if
    $\lambda(a,z) = (f_{a}^n)'(z)$ is locally constant near $(a_0,
    z_0)$ on $M_n$;
  \item \label{def:dh-persistent} \emph{DH-persistent} if for every
    neighborhood $V$ of $z_0$ there exists a neighborhood $W$ of $a_0$
    such that for every $a \in W$, the map $f_a$ has an indifferent
    periodic point of period $n$ in $V$;
  \item \label{def:ms-persistent} \emph{MS-persistent} if there exists
    an analytic map $w$ defined in a neighborhood $U$ of $a_0$ such
    that $w(a_0) = z_0$, $f_a^n(w(a)) = w(a)$, and $|(f_a^n)'(w(a))| =
    1$ for all $a \in U$.
  \end{enumerate}
\end{definition}
Definition (\ref{def:mss-persistent}) is
from the original paper of Ma\~n\'e, Sad, and Sullivan
\cite{ManeSadSullivan1983}, definition (\ref{def:dh-persistent})
is from Douady and Hubbard \cite{DouadyHubbard1985}, and definition
(\ref{def:ms-persistent}) appears both in McMullen's book
\cite{McMullen1994} and in the paper of McMullen and Sullivan
\cite{McMullenSullivan1998}.

In the case where the multiplier $\lambda$ satisfies $\lambda \ne 1$,
these different definitions are actually equivalent, but in the case
$\lambda = 1$ they are not. I am indebted to the referee for pointing
out the following example.

\begin{example}
  Let $f_a(z) = z + z^2 (z-a)^2$. Then for the parameter $a=0$, the
  fixed point $z=0$ is DH-persistent and MS-persistent, but not
  MSS-persistent.
\end{example}
Checking these claims is straightforward. Since every map in the
family $f_a$ has exactly two fixed points, at $z=0$ and $z=a$, both
indifferent with multiplier $\lambda = 1$, the projection $P_1:M_1 \to
\D$ in the definition of MSS-persistence is locally 2-to-1 near
$(0,0)$, so $z=0$ is not MSS-persistent. It is even easier to see that
the conditions in DH-persistence and MS-persistence are
satisfied. Note that the map $f_0$ on its Julia set $J_0$ is not
topologically conjugate to $f_a$ on $J_a$ for $a \ne 0$, since $f_0$
has one fixed point in $J_0$, whereas $f_a$ for $a \ne 0$ has two
fixed points in $J_a$.

All of the later papers refer to \cite{ManeSadSullivan1983} for the
proof of $J$-stability, without giving a proof that the set $S$ (the
interior of the set of parameters for which every indifferent periodic
point is persistent) is the same for these different versions of the
definition. In this paper we chose to use the original definition from
Ma\~n\'e, Sad, and Sullivan, whose proof for rational functions can be
copied directly for polynomial-like maps.

Another subtle point where the literature is inconsistent is the
distinction between stable conjugacy and quasiconformal conjugacy in a
neighborhood of the Julia set. The original proof in
\cite{ManeSadSullivan1983} combined with the extended $\lambda$-lemma
in \cite{SullivanThurston1986} and \cite{BersRoyden1986} gives stable
conjugacy on $S$ for families of rational maps, and the same proof
applies to families of polynomial-like maps. We believe that stable
conjugacy can be promoted to a quasiconformal conjugacy in a
neighborhood of the Julia set (as claimed in
\cite{DouadyHubbard1985}), using the techniques of
\cite{McMullenSullivan1998}, but the situation it a little murky,
especially for families of polynomial-like maps, and we are not aware
of an actual proof of this claim in the literature. In light of this,
and since for our purposes stable conjugacy is sufficient (even though
it makes some arguments a little harder), we chose to work only with
the slightly weaker result which gives stable conjugacies in
neighborhoods of stable parameter values.

\bibliography{everything}

\end{document}